\documentclass[reqno, 11pt]{amsart}
\usepackage[x11names]{xcolor}
\usepackage{amsmath, amssymb}
\usepackage{amsthm}
\usepackage{mathtools}
\usepackage{hyperref}
\usepackage[shortlabels]{enumitem}
\usepackage[portrait, margin=1in]{geometry}

\hypersetup{
    colorlinks = true,
    linktoc = all,
    urlcolor = Blue3!50!Blue4,
    linkcolor = Blue3!50!Blue4,
    citecolor = Blue3!50!Blue4
}

\usepackage{graphicx}
\usepackage[font = small]{caption}

\newtheorem{theorem}{Theorem}[section]
\newtheorem{corollary}[theorem]{Corollary}
\newtheorem{lemma}[theorem]{Lemma}

\newtheorem*{theorem*}{Theorem}
\newtheorem*{corollary*}{Corollary}
\newtheorem*{lemma*}{Lemma}
\newtheorem*{prop*}{Proposition}
\newtheorem*{fact*}{Fact}
\newtheorem*{claim*}{Claim}
\newtheorem*{conjecture*}{Conjecture}

\theoremstyle{definition}
\newtheorem{example}[theorem]{Example}
\newtheorem{definition}[theorem]{Definition}
\newtheorem{remark}[theorem]{Remark}
\newtheorem*{example*}{Example}
\newtheorem*{definition*}{Definition}
\newtheorem*{remark*}{Remark}

\numberwithin{equation}{section}

\newcommand{\tsf}[1]{\mathsf{#1}}

\newcommand{\floor}[1]{\left\lfloor #1 \right\rfloor}
\newcommand{\ceil}[1]{\left\lceil #1 \right\rceil}

\newcommand{\sdbold}[1]{\textbf{\textsf{#1}}}

\newcommand{\abs}[1]{\left\lvert #1 \right\rvert}
\newcommand{\sabs}[1]{\lvert #1 \rvert}

\newcommand{\cC}{\mathcal C}

\newcommand{\cI}{\mathcal I}
\newcommand{\cJ}{\mathcal J}

\newcommand{\ol}{\overline}

\newcommand{\case}[2]{\textcolor{MediumPurple3!80!MediumPurple4}{\sdbold{Case #1} (#2).}}
\newcommand{\casetext}[1]{\textcolor{MediumPurple3!80!MediumPurple4}{\sdbold{#1}}}
\newcommand{\vocab}[1]{\emph{#1}}

\usepackage{tikz}
\usetikzlibrary{trees}
\usetikzlibrary{arrows.meta}
\usetikzlibrary{decorations.pathmorphing}
\tikzset{snaked/.style = {decorate, decoration=snake}}
\usetikzlibrary{positioning}
\usetikzlibrary{decorations.pathreplacing}
\tikzset{sdot/.style = {fill, circle, inner sep = 1.5pt}}
\usetikzlibrary{calc}
\tikzset{every path/.style=thick}

\title{A proof of the Murty--Simon conjecture on diameter-$2$-critical graphs}

\author{Sanjana Das}
\address{Department of Mathematics, Stanford University, Stanford, CA 94305}
\email{sanjanad@stanford.edu}

\author{Jacob Fox}
\address{Department of Mathematics, Stanford University, Stanford, CA 94305}
\email{jacobfox@stanford.edu}
\thanks{JF is supported by NSF award DMS-2452737.}

\date{September 28, 2026}

\begin{document}

\begin{abstract}
    The Murty--Simon conjecture states that every $n$-vertex diameter-$2$-critical graph has at most $\floor{n^2/4}$ edges, with equality only for the balanced complete bipartite graph. Using regularity methods, F\"uredi proved this for all sufficiently large $n$ (specifically, $n$ greater than a tower of $2$'s of height roughly $10^{14}$). We present a proof of the Murty--Simon conjecture that works for all $n$. A proof of the $\floor{n^2/4}$ bound was originally found by AI; we simplified this proof and extended its ideas to get a proof of the full conjecture. 
\end{abstract}

\maketitle
\section{Introduction}

Mantel's theorem \cite{Man07} from 1907 says that every triangle-free graph on $n$ vertices has at most $\lfloor n^2/4 \rfloor$ edges, with equality holding uniquely for the balanced complete bipartite graph $K_{\ceil{n/2}, \floor{n/2}}$. It is considered the first result in extremal graph theory and has been extremely influential. 

In the 1960s and 1970s, several mathematicians proposed a possible strengthening of Mantel's theorem. Instead of working with the class of triangle-free graphs (and looking for the maximum possible number of edges), they worked with the class of \emph{diameter-$2$-critical} graphs --- graphs where every pair of vertices has distance at most $2$, but which lose this property if we delete any edge. 

A short argument shows that every \emph{maximal} triangle-free graph (i.e., a triangle-free graph where adding any new edge would create a triangle) is diameter-$2$-critical. Meanwhile, there are graphs which are diameter-$2$-critical but contain triangles --- for an example, see Figure \ref{fig:diam-2-crit-triangle}. It was conjectured that the conclusion of Mantel's theorem holds even for the class of diameter-$2$-critical graphs --- that the maximum number of edges of an $n$-vertex diameter-$2$-critical graph is still $\floor{n^2/4}$, and the balanced complete bipartite graph $K_{\ceil{n/2}, \floor{n/2}}$ is still the only equality case. In particular, this conjecture would imply Mantel's theorem. 

\begin{figure}[ht]
    \centering
    \begin{tikzpicture}
        \foreach \i in {1, 2, 3, 4, 5, 6} {
            \node [sdot] (\i) at (60*\i + 30:2) {};
            \node at (60*\i + 30:2.35) {$v_\i$};
        }
        \node [sdot, label = above: {$v_7$}] (7) at (0, 0) {};
        \draw [MediumPurple3] (1) -- (3) -- (5) -- (1);
        \foreach \i\j in {1/2, 3/4, 5/6} {
            \draw [MediumPurple3] (\i) -- (\j);
            \draw [MediumPurple3] (7) -- (\j);
        }
    \end{tikzpicture}
    \caption{A diameter-$2$-critical graph containing a triangle $v_1v_3v_5$: Deleting the edge $v_1v_3$ would destroy the only length-$2$ path between $v_2$ and $v_3$ (and likewise with the edges $v_3v_5$ and $v_5v_1$). The other edges are not in triangles, so if we deleted one, there would be no length-$2$ path between its endpoints.} \label{fig:diam-2-crit-triangle}
\end{figure}
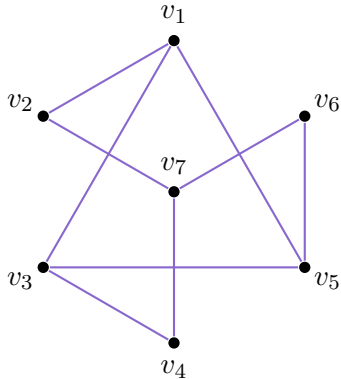

This conjecture has a long history: According to Erd\H{o}s (private communication with F\"uredi --- see \cite{Fur92}), it goes back to work of Ore from the 1960s. In 1975, Plesn\'ik \cite{Ple75} conjectured that any $n$-vertex diameter-$2$-critical graph has at most $\floor{n^2/4}$ edges. This was then independently conjectured by Murty and Simon (see \cite{CH79}), who also conjectured that the balanced complete bipartite graph is the only equality case. This conjecture is now known in the literature as the Murty--Simon conjecture.

In the 1970s and 1980s, several authors made progress towards the Murty--Simon conjecture, proving upper bounds (on the number of edges in an $n$-vertex diameter-$2$-critical graph) of the form $cn^2$ for values of $c$ getting closer and closer to $1/4$ \cite{Ple75, CH79, Fan87}. Then in 1992, in a major breakthrough, F\"uredi \cite{Fur92} proved that the Murty--Simon conjecture holds for all sufficiently large $n$. His proof combined regularity methods (specifically, the triangle removal lemma or the $(6, 3)$-theorem, which are consequences of Szemer\'edi's regularity lemma proven by Ruzsa and Szemer\'edi in \cite{RS78}) and stability-type arguments. 

Interestingly, F\"uredi's proof was the first application of Szemer\'edi's regularity lemma to obtain an \emph{exact} answer  \cite{Fur94}. However, because it uses the regularity lemma, the quantitative bound it gets (specifically, the lower bound on how large $n$ has to be for the proof to work) is enormous --- F\"uredi \cite{Fur92} computed that $n$ has to be at least a tower of $2$'s of height roughly $10^{14}$. (Today, we know slightly better bounds for the triangle removal lemma than the ones coming from Szemer\'edi's regularity lemma --- see \cite{Fox11} --- but these better bounds still involve towers of $2$'s.)

Because the lower bound on $n$ needed by F\"uredi's proof is so enormous, there is still interest in understanding the Murty--Simon conjecture for smaller values of $n$. There is a long line of work proving the Murty--Simon conjecture restricted to various classes of graphs; see \cite{HHMY15} for a nice survey of many of these results. For example, the Murty--Simon conjecture has been proven for graphs with certain forbidden induced subgraphs \cite{HW03, HHY11, HH12a, HH12b, HH14}; graphs with large maximum degree \cite{HHMY14, BJFSSB18}; graphs whose complements have vertex-connectivity $1$, $2$, or $3$ \cite{HHY13, WWY13}; and graphs whose complements have diameter $3$ \cite{HHMY11, Wan12}. There is also work proving stronger bounds for certain restricted classes of graphs \cite{DFH19, LW25}. One possible approach to proving the Murty--Simon conjecture was through a strengthening proposed by Caccetta and H\"aggkvist \cite{CH79}, who conjectured that every diameter-$2$-critical graph has average \emph{edge-degree} at most its number of vertices; however, this conjecture was disproved by Loh and Ma \cite{LM16}.

In this paper, we present a full proof of the Murty--Simon conjecture.

\begin{theorem} \label{thm:murty-simon}
    For all $n \geq 3$, any $n$-vertex diameter-$2$-critical graph has at most $\floor{n^2/4}$ edges. Furthermore, the only graph achieving equality is the balanced complete bipartite graph $K_{\ceil{n/2}, \floor{n/2}}$.
\end{theorem}

We will deduce Theorem \ref{thm:murty-simon} from the following two lemmas (which handle the $\floor{n^2/4}$ bound and the equality case, respectively). 

\begin{lemma} \label{lem:bound}
    Let $G$ be a diameter-$2$-critical graph on $n \geq 3$ vertices, and let $M$ be any maximal matching of $G$. Then \[\abs{E(G) \setminus M} \leq \sabs{\ol{E(G)}}.\] 
\end{lemma}

(We write $E(G)$ and $\ol{E(G)}$ for the sets of edges and non-edges of $G$, respectively. By a \emph{maximal matching}, we mean a matching that we cannot add any additional edge to while keeping it a matching --- i.e., a matching $M$ such that $M \cup \{e\}$ is \emph{not} a matching for each $e \in E(G) \setminus M$.)

\begin{lemma} \label{lem:equality}
    Let $G$ be a diameter-$2$-critical graph on $n \geq 3$ vertices, and suppose that for \emph{every} maximal matching $M$ of $G$, we have \[\abs{E(G) \setminus M} = \sabs{\ol{E(G)}}.\] Then $G$ is triangle-free. 
\end{lemma}

\begin{proof}[Proof of Theorem \ref{thm:murty-simon} assuming Lemmas \ref{lem:bound} and \ref{lem:equality}] 
    Let $G$ be a diameter-$2$-critical graph on $n \geq 3$ vertices, and let $M$ be a maximal matching of $G$ with $m$ edges. Then Lemma \ref{lem:bound} implies that 
    \begin{equation} 
        2\abs{E(G)} \leq \binom{n}{2} + m \leq \binom{n}{2} + \floor{\frac{n}{2}} = 2\floor{\frac{n^2}{4}}. \label{eqn:matching-half-bound}
    \end{equation} 
Dividing by two, we obtain $\sabs{E(G)} \leq \floor{n^2/4}$. 
    
    Furthermore, if equality holds, then we must have equality in \eqref{eqn:matching-half-bound} for \emph{every} maximal matching $M$ of $G$, meaning that \[\abs{E(G) \setminus M} = \sabs{\ol{E(G)}}.\] (We must also have $m = \floor{n/2}$, but we will not use this in our proof.) Then Lemma \ref{lem:equality} implies that $G$ is triangle-free. And by Mantel's theorem, any triangle-free $n$-vertex graph has at most $\floor{n^2/4}$ edges, with equality only achieved by the balanced complete bipartite graph $K_{\ceil{n/2}, \floor{n/2}}$; this gives the desired result. 
\end{proof}

In Section \ref{sec:bound}, we prove the first half of Theorem \ref{thm:murty-simon}, that any $n$-vertex diameter-$2$-critical graph has at most $\floor{n^2/4}$ edges. In Section \ref{sec:equality}, we prove the second half of Theorem \ref{thm:murty-simon}, the characterization of the equality cases.  

\subsection*{AI Statement} This proof was mostly found by AI, specifically GPT-5.6 Sol on August 11, 2026. It found a full proof of the $\floor{n^2/4}$ bound; we simplified some parts of this proof and cleaned up the presentation to get the argument in this paper. For even $n$, it also found a proof that the only equality case is $K_{n/2, n/2}$, based on the fact (proved by Sumner \cite{Sum79}) that the only connected graphs for which every maximal matching is perfect are $K_n$ and $K_{n/2, n/2}$. This fact does not extend nicely to odd $n$ --- there are many examples of graphs for which every maximal matching has $n - 1$ vertices (e.g., any graph with independence number $2$ and an odd number of vertices) --- so we were not able to adapt its proof to handle the equality cases for odd $n$. Instead, we came up with a different proof handling the equality cases which works for all $n$ (which we present in Section \ref{sec:equality}). We wrote up the argument on our own. Before submitting, we also made a final pass of the manuscript through GPT-6 Astra and made a few minor edits based on its feedback. 

\section{Proof of the \texorpdfstring{$\floor{n^2/4}$}{n2/4} bound} \label{sec:bound}

In this section, we prove Lemma \ref{lem:bound}, which implies the first part of the Murty--Simon conjecture (that any $n$-vertex diameter-$2$-critical graph has at most $\floor{n^2/4}$ edges). 

Throughout this section, we use the setup of Lemma \ref{lem:bound} --- we suppose that we are given a diameter-$2$-critical graph $G$ and a maximal matching $M$ of $G$, and we want to show that \[\abs{E(G) \setminus M} \leq \sabs{\ol{E(G)}}.\] 

\subsection{Setup and preliminaries}

First, we define a few terms and objects which we will use in our proof, and prove a few simple claims.

We say a non-edge of $G$ is \emph{critical} if there is only one $2$-path (i.e., a path with $2$ edges) between its endpoints, and we refer to this $2$-path as a \emph{critical $2$-path}. 

\begin{example}
    Some examples of critical and noncritical non-edges: 
    \begin{itemize}
        \item In $K_{n_1, n_2}$ for any $n_1, n_2 \geq 2$, every non-edge is noncritical. 
        \item In the example in Figure \ref{fig:diam-2-crit-triangle}, every non-edge is critical (for example, the non-edge $v_2v_3$ has only the path $(v_2, v_1, v_3)$, the non-edge $v_1v_7$ has only the path $(v_1, v_2, v_7)$, the non-edge $v_2v_4$ has only the path $(v_2, v_7, v_4)$, and so on). 
    \end{itemize}
\end{example}

We will use the following observation in a few places.

\begin{lemma} \label{lem:triangle}
    Suppose that $G$ is diameter-$2$-critical. Then every edge of $G$ in a triangle is also on some critical $2$-path. 
\end{lemma}

For example, in Figure \ref{fig:diam-2-crit-triangle}, the edge $v_1v_3$ is on the critical $2$-path $(v_2, v_1, v_3)$.

\begin{proof}
    Suppose the edge $xy$ is in a triangle. Then if we delete $xy$, there is still a $2$-path between $x$ and $y$ (made by the other two legs of the triangle). So because $G$ is diameter-$2$-critical, there must be some \emph{other} pair of vertices which originally had distance at most $2$ but no longer do after this deletion. The only way this can occur is if that pair is a critical non-edge and $xy$ is on its critical $2$-path. 
\end{proof}

For convenience, we let $E = E(G)$ be the set of edges and $F = \ol{E(G)}$ be the set of non-edges of $G$. We partition $F = F_{\tsf{c}} \cup F_{\tsf{nc}}$, where $F_{\tsf{c}}$ and $F_{\tsf{nc}}$ are the sets of critical and noncritical non-edges, respectively. Given our maximal matching $M$, we also let $E' = E \setminus M$. (So to prove Lemma \ref{lem:bound}, we want to show $\sabs{E'} \leq \abs{F}$.)

We define an auxiliary bipartite graph $\cI$ as follows: We take $\cI$ to have vertex sets $E$ on the left and $F_{\tsf{c}}$ on the right, and we draw an edge $e \sim f$ (for $e \in E$ and $f \in F_{\tsf{c}}$) if $e$ is on the critical $2$-path connecting the endpoints of $f$. We also define \[\cI' = \cI[E', F_{\tsf{c}}]\] to be the graph obtained by deleting $M$ from $\cI$. 

\begin{remark} 
    Note that vertices of $\cI$ are (unordered) \emph{pairs} of vertices in $G$ --- left vertices of $\cI$ are edges of $G$, and right vertices of $\cI$ are critical non-edges of $G$. In particular, $M$ is a subset of the left vertex set of $\cI$ (and $\cI'$ is obtained from $\cI$ by deleting the vertices of $\cI$ which are in $M$). 
\end{remark} 

Note that all vertices on the right side of $\cI$ have degree exactly two --- a critical non-edge has exactly two edges on its critical $2$-path. This means all vertices on the right side of $\cI'$ have degree one or two (the two edges on a critical $2$-path share a vertex, so at most one can be used in $M$). 

Finally, we write $V(M)$ for the set of vertices of $G$ covered by $M$. For $x \in V(M)$, we write $x^*$ for the vertex matched to $x$ in $M$ (meaning that $xx^* \in M$). 

\subsection{Components of \texorpdfstring{$\cI$}{I} and \texorpdfstring{$\cI'$}{Iprime}}

The first step of the proof is to understand the structure of the connected components in $\cI$ and $\cI'$, using the following lemma. (In this section, we will only need to work with components of $\cI'$; but in Section \ref{sec:equality}, we will also work with components of $\cI$.)

\begin{lemma} \label{lem:component-sides}
    Suppose that $C$ is a connected component of either $\cI$ or $\cI'$ with $a$ vertices on the left and $b$ on the right. Then $a \leq b + 1$, with equality if and only if $C$ is a tree whose leaves are all on the left. 
\end{lemma}

\begin{proof}
    On one hand, because $C$ is connected and has $a + b$ vertices, it must have at least $a + b - 1$ edges. On the other hand, all vertices on its right side have degree at most two, so it has at most $2b$ edges. This means \[a + b - 1 \leq 2b,\] and therefore $a \leq b + 1$. For equality to hold, we need $C$ to have exactly $a + b - 1$ edges, meaning that it is a tree; and we need all vertices on its right side to have degree \emph{exactly} two, meaning that all its leaves are on the left. 
\end{proof}

We call a connected component of either $\cI$ or $\cI'$ \emph{left-lopsided} if it has more vertices on the left than the right, \emph{balanced} if it has an equal number of vertices on both sides, and \emph{right-lopsided} if it has more vertices on the right than the left. Lemma \ref{lem:component-sides} states that any left-lopsided component (of either $\cI$ or $\cI'$) has only one more vertex on the left than the right, and is a tree whose leaves are all on the left. Let $\cC$ be the collection of connected components of $\cI'$, and let $\cC_\ell \subseteq \cC$ be the collection of left-lopsided components of $\cI'$. 

\begin{lemma} \label{lem:comp-diffs}
    We have $\sabs{E'} - \abs{F} \leq \abs{\cC_\ell} - \abs{F_{\tsf{nc}}}$, with equality if and only if $\cI'$ has no right-lopsided components. 
\end{lemma}

\begin{proof}
    We can write 
    \begin{equation} 
        \sabs{E'} - \abs{F} = \sabs{E'} - \abs{F_{\tsf{c}}} - \abs{F_{\tsf{nc}}} = \sum_{C \in \cC} (\sabs{C \cap E'} - \sabs{C \cap F_{\tsf{c}}}) - \abs{F_{\tsf{nc}}}\label{eqn:split-sides}
    \end{equation}
    (in words, we are rewriting $\sabs{E'} - \sabs{F_{\tsf{c}}}$ --- the difference in the numbers of vertices on the left and right sides of $\cI'$ --- by splitting $\cI'$ into connected components and looking at the contribution of each component). 

    For each $C \in \cC_\ell$, we have $\sabs{C \cap E'} - \sabs{C \cap F_{\tsf{c}}} = 1$ (by Lemma \ref{lem:component-sides}), while for each $C \in \cC \setminus \cC_\ell$, we have \[\sabs{C \cap E'} - \sabs{C \cap F_{\tsf{c}}} \leq 0\] (this is an equality if $C$ is balanced and a strict inequality if $C$ is right-lopsided). Plugging these bounds into \eqref{eqn:split-sides}, we get \[\sabs{E'} - \abs{F} \leq \abs{\cC_\ell} - \abs{F_{\tsf{nc}}}.\] Equality holds if and only if $\sabs{C \cap E'} - \sabs{C \cap F_{\tsf{c}}} = 0$ for all $C \in \cC \setminus \cC_\ell$, meaning that $\cI'$ has no right-lopsided components. 
\end{proof}

Our goal was to show that $\sabs{E'} \leq \sabs{F}$; by Lemma \ref{lem:comp-diffs}, it suffices to show that $\sabs{\cC_\ell} \leq \sabs{F_{\tsf{nc}}}$. 

\subsection{Switch exits} 

To show that $\abs{\cC_\ell} \leq \abs{F_{\tsf{nc}}}$, we will define a way, based on the given matching $M$, of associating noncritical non-edges to components of $\cI'$. 

\begin{definition}
    For a connected component $C$ of $\cI'$, we say an ordered pair $(x, y) \in V(G) \times V(G)$ is a \vocab{switch exit} for $C$ if $xy$ is a noncritical non-edge in $G$, we have $x \in V(M)$, and $x^*y$ is in $C$. We call such a switch exit \emph{safe} if $y \not\in V(M)$.
\end{definition}

In other words, a switch exit for $C$ is a noncritical non-edge obtained by taking one vertex of $C$ --- which is either an edge or a critical non-edge in $G$ --- and `switching' one of its endpoints using $M$ (i.e., swapping it with its match in $M$). We write switch exits as ordered pairs to keep track of which vertex we switch.  

We define a bipartite multigraph $\cJ$ as follows: We take $\cJ$ to have vertex sets $\cC$ on the left and $F_{\tsf{nc}}$ on the right. For each $C \in \cC$ and $xy \in F_{\tsf{nc}}$, we draw an edge between $C$ and $xy$ if either $(x, y)$ or $(y, x)$ is a switch exit for $C$. If both $(x, y)$ and $(y, x)$ are switch exits for $C$, or if one is a \emph{safe} switch exit for $C$, then we double this edge (i.e., we have two edges between $C$ and $xy$). 

We will show that all vertices on the right side of $\cJ$ (i.e., elements of $F_{\tsf{nc}}$) have degree at most two, while all vertices in $\cC_\ell$ (which is a subset of the left side of $\cJ$) have degree \emph{at least} two; this will imply $\abs{\cC_\ell} \leq \abs{F_{\tsf{nc}}}$. 

\begin{lemma} \label{lem:right-deg-2}
    Every $xy \in F_{\tsf{nc}}$ has degree at most $2$ in $\cJ$. 
\end{lemma}

\begin{proof}
    First note that an \emph{ordered} pair $(x, y)$ is the switch exit for at most one component $C \in \cC$, as $x^*y$ cannot be in more than one connected component of $\cI'$. 

    So for an unordered pair $xy \in F_{\tsf{nc}}$, if both $x$ and $y$ are in $V(M)$, then $xy$ has at most one edge in $\cJ$ corresponding to $(x, y)$ and at most one corresponding to $(y, x)$ (as neither of these are safe). Meanwhile, if $x \in V(M)$ and $y \not\in V(M)$, then $xy$ has at most two edges in $\cJ$ corresponding to $(x, y)$ (which may be a safe switch exit of one component) and none corresponding to $(y, x)$. And if $x$ and $y$ are both not in $V(M)$, then $xy$ has no edges in $\cJ$. 
\end{proof}

The harder part of the proof is showing that each $C \in \cC_\ell$ has degree at least two in $\cJ$, meaning that $C$ has either a safe switch exit or at least two switch exits. We will show this in a few steps. 

\begin{remark} 
    In illustrations, we draw edges in $E'$ in \textcolor{MediumPurple3}{purple}, edges in $M$ in \textcolor{SkyBlue3}{blue}, and non-edges as dashed lines in \textcolor{SpringGreen4}{green}. In Section \ref{sec:equality}, for illustrations that occur before specifying $M$, we will draw all edges in \textcolor{gray}{gray} (and we will still draw non-edges in the same way).  
\end{remark}

When finding these switch exits, we will often make use of the following fact, which says that for $C \in \cC_\ell$, any non-edge obtained by switching an \emph{edge} belonging to $C$ (i.e., a left vertex of $C$) is automatically noncritical, and therefore a switch exit for $C$. 

\begin{lemma} \label{lem:auto-noncrit}
    Let $C$ be a left-lopsided component of $\cI'$. Suppose that $xy$ is a non-edge of $G$, we have $x \in V(M)$, and $x^*y$ is an edge of $G$ which is in $C$. Then $xy$ is noncritical. 
\end{lemma}

\begin{proof}
    Suppose for contradiction that $xy$ is a critical non-edge in $G$. Then $(x, x^*, y)$ is its critical $2$-path. So $xy$ is a vertex on the right side of $\cI'$, and its only neighbor in $\cI'$ is $x^*y$ (as $xx^*$ is in $M$, so it is not a vertex in $\cI'$). 

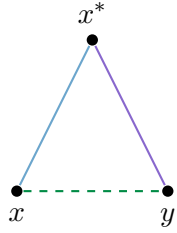
    \begin{center}
        \begin{tikzpicture}
            \node [sdot, label = below: {$x$}] (x) at (-1, 0) {};
            \node [sdot, label = below: {$y$}] (y) at (1, 0) {};
            \node [sdot, label = above: {$x^*$}] (xst) at (0, 2) {};
            \draw [MediumPurple3] (xst) -- (y);
            \draw [SkyBlue3] (xst) -- (x);
            \draw [SpringGreen4, dashed] (x) -- (y);
        \end{tikzpicture}
    \end{center}

    Since $C$ is the component of $\cI'$ containing $x^*y$, it also contains $xy$. But $xy$ is a leaf in $C$, which contradicts Lemma \ref{lem:component-sides} (since $C$ is left-lopsided, it cannot have any leaves on the right side). 
\end{proof}

We first handle the case where $C$ consists of a single vertex on the left side of $\cI'$ (i.e., a single edge of $G$). 

\begin{lemma} \label{lem:single-vtx-exits}
    Suppose that $C$ is a connected component of $\cI'$ consisting of a single vertex on the left side of $\cI'$ (i.e., a single edge of $G$). Then $C$ has either a safe switch exit or two switch exits. 
\end{lemma}

\begin{proof}
    Let $C$ consist of the single edge $xy$ of $G$. Because $M$ is a maximal matching, at least one of $x$ and $y$ must be in $V(M)$ (otherwise we could add $xy$ to $M$); suppose without loss of generality that $x \in V(M)$. (Note that $x^* \neq y$ --- $xy$ cannot be in $M$, as it is in $\cI'$.)

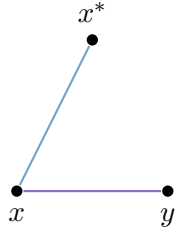
    \begin{center}
        \begin{tikzpicture}
            \node [sdot, label = below: {$x$}] (x) at (-1, 0) {};
            \node [sdot, label = below: {$y$}] (y) at (1, 0) {};
            \node [sdot, label = above: {$x^*$}] (xst) at (0, 2) {};
            \draw [MediumPurple3] (x) -- (y);
            \draw [SkyBlue3] (x) -- (xst);
        \end{tikzpicture}
    \end{center}

    First, we claim that $x^*y$ is a non-edge in $G$ --- if $x^*y$ were an edge in $G$, then $xy$ would be in a triangle, so by Lemma \ref{lem:triangle}, it would have to be on some critical $2$-path. Then $xy$ would have a neighbor in $\cI'$ (the critical non-edge corresponding to that critical $2$-path), contradicting our assumption that $xy$ forms a single-vertex connected component. 

    
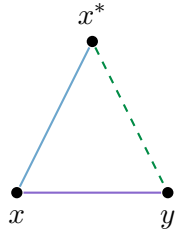
\begin{center}
        \begin{tikzpicture}
            \node [sdot, label = below: {$x$}] (x) at (-1, 0) {};
            \node [sdot, label = below: {$y$}] (y) at (1, 0) {};
            \node [sdot, label = above: {$x^*$}] (xst) at (0, 2) {};
            \draw [MediumPurple3] (x) -- (y);
            \draw [SkyBlue3] (x) -- (xst);
            \draw [SpringGreen4, dashed] (xst) -- (y);
        \end{tikzpicture}
    \end{center}

    Now, by Lemma \ref{lem:auto-noncrit}, $x^*y$ must be noncritical. So $(x^*, y)$ is a switch exit for $C$ (it is a noncritical non-edge which switches to $xy$, which is in $C$). 

    Now, if $y \not\in V(M)$, then $(x^*, y)$ is a safe switch exit. Meanwhile, if $y \in V(M)$, then the same argument (with the roles of $x$ and $y$ swapped) shows that $(y^*, x)$ is also a switch exit for $C$, so $C$ has two switch exits. 
\end{proof}

We now handle the case where $C$ consists of more than one vertex. 

\begin{lemma} \label{lem:big-component-exits}
    Suppose that $C$ is a left-lopsided component of $\cI'$ consisting of more than one vertex of $\cI'$. Then $C$ has either a safe switch exit, or at least two switch exits. 
\end{lemma}

\begin{proof}
    We will first focus on showing that we can always find \emph{one} switch exit for $C$, and then show that this proof actually gives a safe switch exit or two switch exits. 

    By Lemma \ref{lem:component-sides}, $C$ is a tree whose leaves are all on the left. Let $uy$ be a leaf of $C$ on the left, and let its one neighbor in $\cI'$ be $xy$ (so $xy$ is a critical non-edge in $G$, and $(x, u, y)$ is its critical $2$-path). Then $ux$ must also be in $C$ (i.e., it is not in $M$), as otherwise $xy$ would be a leaf of $C$ on the right. 

    
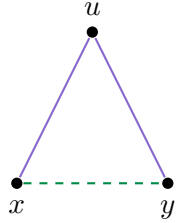
\begin{center}
        \begin{tikzpicture}
            \node [sdot, label = below: {$x$}] (x) at (-1, 0) {};
            \node [sdot, label = below: {$y$}] (y) at (1, 0) {};
            \node [sdot, label = above: {$u$}] (u) at (0, 2) {};
            \draw [MediumPurple3] (x) -- (u) -- (y);
            \draw [SpringGreen4, dashed] (x) -- (y);
        \end{tikzpicture}
    \end{center}

    Because $M$ is a maximal matching, either $u$ or $x$ must be in $V(M)$ --- otherwise we could add $ux$ to $M$. We will perform casework based on which of $u$ and $x$ is in $V(M)$ (note that both may be in $V(M)$, so \casetext{Case 1} and \casetext{Case 2} may both occur). 

    \case{1}{$u \in V(M)$} Then we have the following picture. 

    
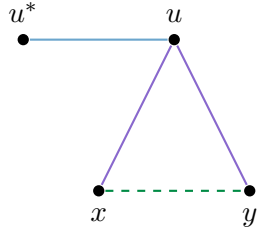
\begin{center}
        \begin{tikzpicture}
            \node [sdot, label = below: {$x$}] (x) at (-1, 0) {};
            \node [sdot, label = below: {$y$}] (y) at (1, 0) {};
            \node [sdot, label = above: {$u$}] (u) at (0, 2) {};
            \node [sdot, label = above: {$u^*$}] (ust) at (-2, 2) {};
            \draw [MediumPurple3] (x) -- (u) -- (y);
            \draw [SpringGreen4, dashed] (x) -- (y);
            \draw [SkyBlue3] (u) -- (ust);
        \end{tikzpicture}
    \end{center}
    
    Because $xy$ is critical, $u^*x$ and $u^*y$ cannot both be edges (if they were, then $(x, u^*, y)$ would be a second $2$-path between $x$ and $y$, in addition to $(x, u, y)$). 

    \case{1a}{$u^*x$ is a non-edge} Then $u^*x$ is noncritical by Lemma \ref{lem:auto-noncrit}, as it is a non-edge which switches to $ux$ (which is an edge of $G$ belonging to $C$). So $(u^*, x)$ is a switch exit for $C$. 

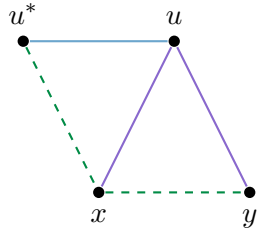
    \begin{center}
        \begin{tikzpicture}
            \node [sdot, label = below: {$x$}] (x) at (-1, 0) {};
            \node [sdot, label = below: {$y$}] (y) at (1, 0) {};
            \node [sdot, label = above: {$u$}] (u) at (0, 2) {};
            \node [sdot, label = above: {$u^*$}] (ust) at (-2, 2) {};
            \draw [MediumPurple3] (x) -- (u) -- (y);
            \draw [SpringGreen4, dashed] (x) -- (y);
            \draw [SpringGreen4, dashed] (ust) -- (x);
            \draw [SkyBlue3] (u) -- (ust);
        \end{tikzpicture}
    \end{center}

    \case{1b}{$u^*y$ is a non-edge} Then the same reasoning shows that $u^*y$ is noncritical, so $(u^*, y)$ is a switch exit for $C$. 

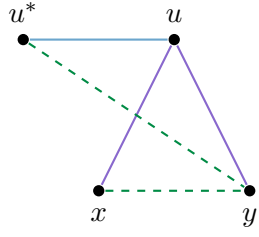
    \begin{center}
        \begin{tikzpicture}
            \node [sdot, label = below: {$x$}] (x) at (-1, 0) {};
            \node [sdot, label = below: {$y$}] (y) at (1, 0) {};
            \node [sdot, label = above: {$u$}] (u) at (0, 2) {};
            \node [sdot, label = above: {$u^*$}] (ust) at (-2, 2) {};
            \draw [MediumPurple3] (x) -- (u) -- (y);
            \draw [SpringGreen4, dashed] (x) -- (y);
            \draw [SpringGreen4, dashed] (ust) -- (y);
            \draw [SkyBlue3] (u) -- (ust);
        \end{tikzpicture}
    \end{center}

    \case{2}{$x \in V(M)$} Then we have the following picture. 

    
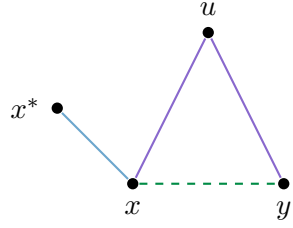
\begin{center}
        \begin{tikzpicture}
            \node [sdot, label = below: {$x$}] (x) at (-1, 0) {};
            \node [sdot, label = below: {$y$}] (y) at (1, 0) {};
            \node [sdot, label = above: {$u$}] (u) at (0, 2) {};
            \node [sdot, label = left: {$x^*$}] (xst) at (-2, 1) {};
            \draw [MediumPurple3] (x) -- (u) -- (y);
            \draw [SpringGreen4, dashed] (x) -- (y);
            \draw [SkyBlue3] (xst) -- (x);
        \end{tikzpicture}
    \end{center}

    We now perform casework based on whether $x^*u$ is an edge or non-edge. 

    \case{2a}{$x^*u$ is a non-edge} Then $x^*u$ is noncritical by Lemma \ref{lem:auto-noncrit} (it is a non-edge which switches to $xu$, which is an edge of $G$ belonging to $C$), so $(x^*, u)$ is a switch exit of $C$. 

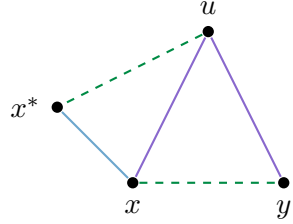
    \begin{center}
        \begin{tikzpicture}
            \node [sdot, label = below: {$x$}] (x) at (-1, 0) {};
            \node [sdot, label = below: {$y$}] (y) at (1, 0) {};
            \node [sdot, label = above: {$u$}] (u) at (0, 2) {};
            \node [sdot, label = left: {$x^*$}] (xst) at (-2, 1) {};
            \draw [MediumPurple3] (x) -- (u) -- (y);
            \draw [SpringGreen4, dashed] (x) -- (y);
            \draw [SkyBlue3] (xst) -- (x);
            \draw [SpringGreen4, dashed] (u) -- (xst);
        \end{tikzpicture}
    \end{center}

    \case{2b}{$x^*u$ is an edge} First note that $x^*y$ must be a non-edge, as otherwise $(x, x^*, y)$ would be a second $2$-path connecting $x$ and $y$ (contradicting the criticality of $xy$). 

    
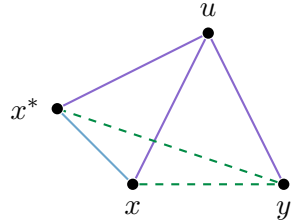
\begin{center}
        \begin{tikzpicture}
            \node [sdot, label = below: {$x$}] (x) at (-1, 0) {};
            \node [sdot, label = below: {$y$}] (y) at (1, 0) {};
            \node [sdot, label = above: {$u$}] (u) at (0, 2) {};
            \node [sdot, label = left: {$x^*$}] (xst) at (-2, 1) {};
            \draw [MediumPurple3] (x) -- (u) -- (y);
            \draw [SpringGreen4, dashed] (x) -- (y) -- (xst);
            \draw [SkyBlue3] (xst) -- (x);
            \draw [MediumPurple3] (u) -- (xst);
        \end{tikzpicture}
    \end{center}
    
    Now, we claim that $x^*y$ is noncritical. Assume for contradiction that $x^*y$ is a critical non-edge, as in the above picture. Then in $\cI'$, we have the adjacencies \[ux \sim xy \sim uy \sim x^*y \sim ux^*\] (we write $ux \sim xy$ to mean that $ux$ and $xy$ are adjacent in $\cI'$; this is because $ux$ and $uy$ are on the critical $2$-path for $xy$, and $ux^*$ and $uy$ are on the critical $2$-path for $x^*y$). This contradicts our assumption that $uy$ was a leaf in $C$ (as it has at least two neighbors in $\cI'$, and therefore also $C$). 

    This shows $x^*y$ must be noncritical. Then $x^*y$ is a noncritical non-edge which switches to $xy$, which is a non-edge of $G$ belonging to $C$; this means $(x^*, y)$ is a switch exit for $C$. 

    So far, this argument shows that in all cases, we can always find at least \emph{one} switch exit for $C$. Now we will show that we can find either one safe switch exit or two switch exits. 

    First note that if we are in \casetext{Case 1a}, then either $(u^*, x)$ is a safe switch exit, or $x$ is also in $V(M)$, in which case we can get a second switch exit from \casetext{Case 2}. Similarly, if we are in \casetext{Case 2a}, then either $(x^*, u)$ is a safe switch exit, or $u$ is also in $V(M)$, in which case we get a second switch exit from \casetext{Case 1}. 

    Now, because $C$ is a tree with more than one vertex, it has at least two leaves, and we can run this argument starting from all of those leaves (recall that we began the argument by choosing a leaf $uy$ of $C$). If any leaf lands in \casetext{Case 1a} or \casetext{Case 2a}, then we win, as explained above (we get either a safe switch exit or two switch exits just from that leaf). So we only need to worry about the case where all leaves land in \casetext{Case 1b} and \casetext{Case 2b}. 

    Also, if one leaf lands in \casetext{Case 1b} and another lands in \casetext{Case 2b}, then they produce different switch exits --- this is because a switch exit produced by \casetext{Case 1b} switches to an \emph{edge} belonging to $C$, while a switch exit produced by \casetext{Case 2b} switches to a \emph{non-edge} belonging to $C$. So it remains only to consider the case where all leaves land in \casetext{Case 1b} and the case where all leaves land in \casetext{Case 2b}. 

    Next, we claim that two different leaves in \casetext{Case 1b} produce different switch exits. This is because if a leaf produces a switch exit $(u^*, y)$ through \casetext{Case 1b}, then this switch exit switches back to the leaf --- i.e., the leaf must have been $uy$. So we can recover the identity of that leaf knowing the switch exit (and that it was produced by \casetext{Case 1b}). 

    Similarly, we claim that two different leaves in \casetext{Case 2b} produce different switch exits --- i.e., given that a leaf produces a switch exit $(x^*, y)$ through \casetext{Case 2b}, we can recover the identity of that leaf. This is because we know $xy$ (obtained by switching $x^*$ to $x$) is the unique neighbor of our starting leaf. And $xy$ must be critical, so it has a unique critical $2$-path $(x, u, y)$. This recovers $u$, and we now know that the leaf must have been $uy$ (i.e., the edge of this critical $2$-path using the second vertex of our switch exit). 

    This shows that we can always find either a safe switch exit for $C$ or at least two switch exits, as desired. 
\end{proof}

Combining Lemmas \ref{lem:single-vtx-exits} and \ref{lem:big-component-exits} (and noting that a safe switch exit corresponds to a doubled edge in $\cJ$) gives the following. 

\begin{corollary} \label{cor:left-deg-2}
    Every $C \in \cC_\ell$ has degree at least $2$ in $\cJ$. 
\end{corollary}

We can now combine this with Lemma \ref{lem:right-deg-2} (which states that every element of $F_{\tsf{nc}}$, i.e., every vertex on the right side of $\cJ$, has degree \emph{at most} $2$) to get the following. 

\begin{lemma} \label{lem:j-sides}
    We have $\abs{\cC_\ell} \leq \abs{F_{\tsf{nc}}}$. Furthermore, if equality holds, then no component $C \in \cC \setminus \cC_\ell$ can have a switch exit. 
\end{lemma}

\begin{proof}
    Recall that $\cJ$ is a bipartite multigraph with vertex sets $\cC$ on the left and $F_{\tsf{nc}}$ on the right. Lemma \ref{lem:right-deg-2} states that all elements of $F_{\tsf{nc}}$ have degree at most $2$ in $\cJ$, so \[e(\cJ) \leq 2\abs{F_{\tsf{nc}}}.\] On the other hand, Corollary \ref{cor:left-deg-2} states that all elements of $\cC_\ell \subseteq \cC$ have degree at least $2$ in $\cJ$, so \[e(\cJ) \geq 2\abs{\cC_\ell}.\] Combining these, we get that $\abs{\cC_\ell} \leq \abs{F_{\tsf{nc}}}$. For equality to hold, we must have $e(\cJ) = 2\abs{\cC_\ell}$, which means each $C \in \cC_\ell$ has degree \emph{exactly} $2$ in $\cJ$ and each $C \in \cC \setminus \cC_\ell$ has degree $0$, meaning that it has no switch exits. 
\end{proof}

Combining Lemmas \ref{lem:comp-diffs} and \ref{lem:j-sides} completes the proof of Lemma \ref{lem:bound} (and gives a few conditions that must hold in an equality case, which we will use in our proof of Lemma \ref{lem:equality} in the next section). 

\begin{corollary} \label{cor:bound}
    We have $\abs{E'} \leq \abs{F}$. Furthermore, if equality holds, then $\cI'$ does not contain any right-lopsided components, and no balanced component of $\cI'$ has a switch exit. 
\end{corollary}

\section{Characterization of the equality cases} \label{sec:equality}

In this section, we prove Lemma \ref{lem:equality}, which states that if $G$ is a diameter-$2$-critical graph for which equality holds in Lemma \ref{lem:bound} for \emph{every} maximal matching $M$ (we will refer to such a graph $G$ as an \emph{equality case}), then $G$ is triangle-free. We will actually prove the following slightly stronger statement. 

\begin{lemma} \label{lem:strong-equality}
    If $G$ is an equality case, then either $G$ is a star $K_{1, n - 1}$, or $G$ has no critical non-edges. 
\end{lemma}

(This implies Lemma \ref{lem:equality} because by Lemma \ref{lem:triangle}, any edge in a triangle is also on some critical $2$-path; so if $G$ has no critical non-edges (and therefore no critical $2$-paths), it cannot have any triangles.)

We continue to use the notation from Section \ref{sec:bound}. (Note that several objects defined in Section \ref{sec:bound}, such as $\cI'$ and the notion of a switch exit, implicitly depend on $M$.) We will use the properties of an equality case described in Corollary \ref{cor:bound} --- we will show that if $G$ contains any critical non-edge and is not a star, then either we can find a maximal matching $M$ for which $\cI'$ has a right-lopsided component, or one where some balanced component of $\cI'$ has a switch exit. 

\begin{lemma} \label{lem:all-comps-left-lops}
    If $G$ is an equality case, then every connected component of $\cI$ is left-lopsided, and for each component $C$ of $\cI$, every matching in $G$ uses at most one left vertex of $C$.
\end{lemma}

(Recall that the left vertices of $\cI$, and therefore any component $C$ of $\cI$, are edges in $G$.)

\begin{proof}
    If we track what happens to a component $C$ of $\cI$ when we move to $\cI'$, we delete its left vertices which are in $M$; this may split its remainder into several components of $\cI'$. If $C$ had $a$ vertices on the left and $b$ on the right, and $c$ of its left vertices are in $M$, then the components $C$ splits into have a total of $a - c$ vertices on the left and $b$ on the right. If $G$ is an equality case, then each of those components must be left-lopsided or balanced by Corollary \ref{cor:bound}, so we must have $a - c \geq b$ (for every $M$).

    If $C$ is right-lopsided (so $a < b$), then $a - c < b$ for \emph{every} $M$. If $C$ is balanced (so $a = b$), then we can construct $M$ with $c \geq 1$ by placing one left vertex of $C$ (which is an edge of $G$) into $M$, and greedily extending it to a maximal matching (repeatedly adding edges of $G$ to $M$ until we get stuck); then $a - c < b$ for this choice of $M$. Finally, if $C$ is left-lopsided (so $a = b + 1$) and some matching uses more than one left vertex of $C$, then we can greedily extend this to a maximal matching $M$ (which has $c \geq 2$); so we also have $a - c < b$ for this choice of $M$. 
\end{proof}

\begin{lemma} \label{lem:no-big-comps}
    If $G$ is an equality case and $\cI$ has any left-lopsided component $C$ with at least one vertex on the right, then $G$ is a star $K_{1, n - 1}$. 
\end{lemma}

\begin{proof}
    By Lemma \ref{lem:component-sides}, $C$ must be a tree all of whose leaves are on the left. Let $uy$ be a leaf of $C$ on the left and let $xy$ be its one neighbor on the right; then $ux$ is the other neighbor of $xy$ on the left (this means $xy$ is a critical non-edge and $(x, u, y)$ is its critical $2$-path). 

    
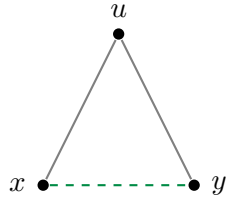
\begin{center}
        \begin{tikzpicture}
            \node [sdot, label = left: {$x$}] (x) at (-1, 0) {};
            \node [sdot, label = right: {$y$}] (y) at (1, 0) {};
            \node [sdot, label = above: {$u$}] (u) at (0, 2) {};
            \draw [gray] (x) -- (u) -- (y);
            \draw [SpringGreen4, dashed] (x) -- (y);
        \end{tikzpicture}
    \end{center}

    We split the rest of the proof into three cases, one of which must always occur. 

    \case{1}{$x$ has a neighbor $v \neq u$ which is not a neighbor of $u$} In this case, we have the following picture.  

    
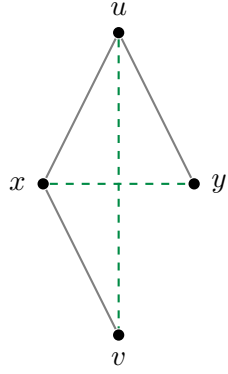
\begin{center}
        \begin{tikzpicture}
            \node [sdot, label = left: {$x$}] (x) at (-1, 0) {};
            \node [sdot, label = right: {$y$}] (y) at (1, 0) {};
            \node [sdot, label = above: {$u$}] (u) at (0, 2) {};
            \node [sdot, label = below: {$v$}] (v) at (0, -2) {};
            \draw [gray] (x) -- (u) -- (y);
            \draw [SpringGreen4, dashed] (x) -- (y);
            \draw [gray] (x) -- (v);
            \draw [SpringGreen4, dashed] (u) -- (v);
        \end{tikzpicture}
    \end{center}

    If the non-edge $uv$ is critical, then in $\cI$ we have the adjacencies \[uy \sim xy \sim ux \sim uv \sim xv\] (since $ux$ and $uy$ are on the critical $2$-path for $xy$, while $ux$ and $xv$ are on the critical $2$-path for $uv$). This means $uy$ and $xv$ are both in $C$; this contradicts Lemma \ref{lem:all-comps-left-lops}, as the matching $\{uy, xv\}$ in $G$ uses two edges belonging to $C$. 

    So $uv$ must be noncritical. Then we can construct $M$ by placing $uy$ and $xv$ in it and greedily extending this to a maximal matching. By Lemma \ref{lem:all-comps-left-lops}, no other vertices of $C$ (other than $uy$) can land in $M$; so $C$ turns into a balanced component $C'$ of $\cI'$ obtained by just deleting $uy$ (which was a leaf in $C$). And $(v, u)$ is a switch exit for $C'$, as it switches to $xu$ (which is in $C'$). So we have constructed $M$ for which $\cI'$ has a balanced component with a switch exit; this contradicts Corollary \ref{cor:bound} (which implies that if $G$ is an equality case, no balanced component can have a switch exit). 

    
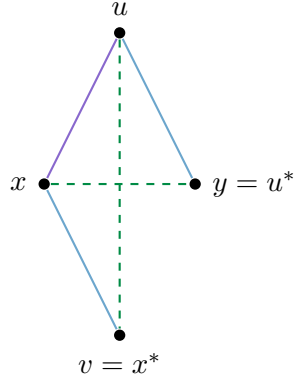
\begin{center}
        \begin{tikzpicture}
            \node [sdot, label = left: {$x$}] (x) at (-1, 0) {};
            \node [sdot, label = right: {$y = u^*$}] (y) at (1, 0) {};
            \node [sdot, label = above: {$u$}] (u) at (0, 2) {};
            \node [sdot, label = below: {$v = x^*$}] (v) at (0, -2) {};
            \draw [MediumPurple3] (x) -- (u);
            \draw [SkyBlue3] (u) -- (y);
            \draw [SpringGreen4, dashed] (x) -- (y);
            \draw [SkyBlue3] (x) -- (v);
            \draw [SpringGreen4, dashed] (u) -- (v);
        \end{tikzpicture}
    \end{center}

    \case{2}{$x$ and $u$ have a common neighbor $v$} Then because $xy$ is critical, $yv$ must be a non-edge (or else $(x, v, y)$ would be a second $2$-path between $x$ and $y$). 

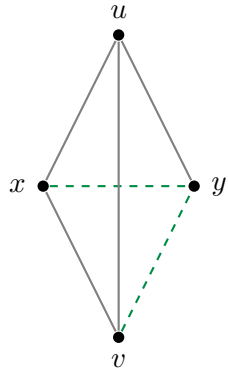
    \begin{center}
        \begin{tikzpicture}
            \node [sdot, label = left: {$x$}] (x) at (-1, 0) {};
            \node [sdot, label = right: {$y$}] (y) at (1, 0) {};
            \node [sdot, label = above: {$u$}] (u) at (0, 2) {};
            \node [sdot, label = below: {$v$}] (v) at (0, -2) {};
            \draw [gray] (x) -- (u) -- (y);
            \draw [SpringGreen4, dashed] (x) -- (y);
            \draw [gray] (x) -- (v);
            \draw [gray] (u) -- (v);
            \draw [SpringGreen4, dashed] (v) -- (y);
        \end{tikzpicture}
    \end{center}

    If $yv$ is critical, then $\cI$ has the adjacencies \[ux \sim xy \sim uy \sim vy \sim uv,\] contradicting the assumption that $uy$ is a leaf in $C$. 

    So $yv$ must be noncritical. Then we can construct $M$ by placing $xv$ and $uy$ in it and greedily extending this to a maximal matching. Again $uy$ must be the only vertex of $C$ placed into $M$ by Lemma \ref{lem:all-comps-left-lops}, so $C$ turns into a balanced component $C'$ of $\cI'$ obtained by just deleting $uy$. And $(v, y)$ is a switch exit for $C'$ because it switches to $xy$, which is in $C'$. So we have again obtained a balanced component with a switch exit, which contradicts Corollary \ref{cor:bound}. 

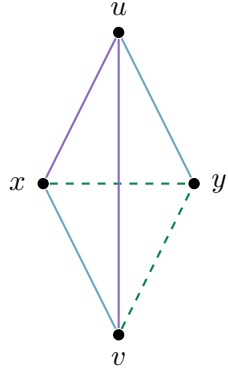
    \begin{center}
        \begin{tikzpicture}
            \node [sdot, label = left: {$x$}] (x) at (-1, 0) {};
            \node [sdot, label = right: {$y$}] (y) at (1, 0) {};
            \node [sdot, label = above: {$u$}] (u) at (0, 2) {};
            \node [sdot, label = below: {$v$}] (v) at (0, -2) {};
            \draw [MediumPurple3] (x) -- (u);
            \draw [SpringGreen4, dashed] (x) -- (y);
            \draw [SkyBlue3] (x) -- (v);
            \draw [SkyBlue3] (u) -- (y);
            \draw [MediumPurple3] (u) -- (v);
            \draw [SpringGreen4, dashed] (v) -- (y);
        \end{tikzpicture}
    \end{center}

    \case{3}{$x$ has no neighbors other than $u$} Then because $G$ has diameter $2$, every other vertex must be adjacent to $u$ (in order to have distance $2$ to $x$); this means $G$ contains a star centered at $u$ using all the vertices of $G$. But such a star already has diameter $2$; since $G$ is diameter-$2$-critical, this means $G$ cannot contain any other edges (as such an edge could be deleted). 
\end{proof}

Combining Lemmas \ref{lem:all-comps-left-lops} and \ref{lem:no-big-comps}, we get that if $G$ is an equality case, then either $G$ is a star, or $\cI$ cannot have any vertices on the right (if $\cI$ had a vertex on the right, it would be contained in some connected component, which would have to be left-lopsided by Lemma \ref{lem:all-comps-left-lops}), meaning that $G$ cannot have any critical non-edges; this completes the proof of Lemma \ref{lem:strong-equality}, and therefore also Lemma \ref{lem:equality}. 

\newpage

\bibliography{refs}{}
\bibliographystyle{plain}

\end{document}